\documentclass[english,11pt,reqno]{smfart}

\usepackage{color}
\usepackage{amsthm}
\usepackage{amsmath}
\usepackage{amsfonts}
\usepackage{amstext}
\usepackage[latin1]{inputenc}
\usepackage{amscd}
\usepackage{latexsym}
\usepackage{bm}
\usepackage{amssymb}
\usepackage[all,cmtip]{xy}
\usepackage{euscript}
\usepackage{a4wide}
\usepackage{amsmath,amssymb,graphicx}
\usepackage{amssymb}
\usepackage{amsmath}
\usepackage{dsfont}

\newtheorem{theorem}{Theorem}
\newtheorem{proposition}[theorem]{Proposition}
\newtheorem{lemma}[theorem]{Lemma}

\newtheorem{definition}[theorem]{Definition}
\newtheorem{remark}{Remark}

\newcommand{\N}{\mathbb{N}}
\newcommand{\R}{\mathbb{R}}
\newcommand{\TT}{\mathbb{T}}

\newcommand{\Z}{\mathbb{Z}}
\newcommand{\di}{\displaystyle}

\newcommand{\fonction}[5]{\begin{array}[t]{lrcl}#1 :&#2 &\longrightarrow &#3\\&#4& \longmapsto &#5 \end{array}}

\newcommand{\fonctionsansdef}[3]{\begin{array}[t]{lrcl}#1 :&#2 &\longrightarrow &#3 \end{array}}

\date{}

\begin{document}
\setcounter{tocdepth}{3}
\title[Discrete Hamiltonian systems and Helmholtz conditions]{Continuous versus discrete structures II -- Discrete Hamiltonian systems and Helmholtz conditions}
\author{Jacky Cresson$^{1,2}$ and Fr\'ed\'eric Pierret$^1$}
\address{$^1$ SYRTE, Observatoire de Paris, 77 avenue Denfert-Rochereau, 75014 Paris, France}

\begin{abstract}
We define discrete Hamiltonian systems in the framework of discrete embeddings. An explicit comparison with previous attempts is given. We then solve the discrete Helmholtz's inverse problem for the discrete calculus of variation in the Hamiltonian setting. Several applications are discussed.
\end{abstract}

\maketitle

\noindent {\tiny $^1$ SYRTE, Observatoire de Paris, 77 avenue Denfert-Rochereau, 75014 Paris, France}

\noindent {\tiny $^2$ Laboratoire de Math\'ematiques Appliqu\'ees de Pau, Universit\'e de Pau et des Pays de l'Adour,}

\noindent {\tiny  avenue de l'Universit\'e, BP 1155, 64013 Pau Cedex, France}

\tableofcontents

\textbf{\textrm{Keywords:}} Discrete Helmholtz condition; discrete calculus of variations, discrete embedding.

\section{Introduction}

In recent years, many efforts have been devoted to the definition of discrete analogue of continuous methods and objects, like variational calculus, Lagrangian and Hamiltonian systems, etc. We refer to \cite{lubisch} and \cite{marsden} for an overview of this subject. Although different points of view can be adopted, a common point of all these works is to understand in which extend a discrete object will differ from the original continuous one, as for example symmetries, first integrals and related properties. 

Following our previous work \cite{cresson-bourdin}, we continue to investigate the discrete calculus of variations in the formulation given for example in \cite{cresson-bourdin} following the classical work of \cite{marsden} and \cite{lubisch}. This formulation, which is obtained using the formalism of discrete embedding (see \cite{cp1}), has the property to mimic very closely the continuous case. As an example of this difference in the formulation, on a can compare the discrete Euler-Lagrange equation given in \cite{lubisch} and the one given in \cite{cresson-bourdin} or \cite{cp1}.

In \cite{cresson-bourdin}, we have derived the discrete Helmholtz's condition for second order difference equations, i.e. conditions under which such equations can be written as a discrete Euler-Lagrange equation. In this paper, we want to discuss under which condition a set of difference equations can be written as a discrete Hamiltonian system. This subject is studied in \cite{opri1} and \cite{opri2} where they obtain some conditions. However, the underlying framework concerning the variational integrators is not easily related the continuous one, as well as the proofs which are given. \\

In this paper, we first derive in the framework of \emph{discrete embedding}, a coherent definition of \emph{discrete Hamiltonian systems}. Comparison of our presentation with previous attempts, in particular (\cite{marsden},\cite{lall},\cite{lubisch}) will be given. Second, we derive the Helmholtz conditions in the Hamiltonian case following the usual self adjoin characterization of the differential operator associated to the Hamiltonian system as exposed for example by Santilli \cite{santilli}.

\part{Notations}

We remind several notations and results developed in \cite{cp1} in the framework of the \emph{discrete embedding}. We refer to this article for more details and proofs.\\

Let $N\in\N$ and $a,b\in \R$ with $a<b$ and let $h=(b-a)/N$. We denote by $\TT$ the subspace of $\R$ defined by $\TT=h\Z \cap [a,b]$ where $h\Z=\{hz | z\in\R \}$. The elements of $\TT$ are denotes by $t_k$ for $k=0,...,N$. We denote by $\TT^+ =\TT \setminus \{ t_N \}$, $\TT^- =\TT \setminus \{ t_0 \}$ and $\TT^{\pm} =\TT^+ \cap \TT^-$.\\

We denote by $C([a,b],\mathbb{R}^{d})$ the set of functions $x :[a,b]\rightarrow \R^d$, $d\in \N^{*}$ and by $C^{i}([a,b],\ \R^{d})$ the set of i-th differentiable functions.\\

We denote $C(\TT,\R^d )$ the set of functions with value in $\R^d $ over $\TT$ and $C_0 (\TT ,\R^d )$ the subset defined by
\begin{equation}
C_0(\TT,\R^d )=\{G\in C(\TT,\R^d),\ G_{0}=G_{N}=0\}.
\end{equation}

The discrete derivatives $\Delta $ and $\nabla$ defined over are respectively defined by the forward and backward finite differences operator with value in $C(\TT^+ ,\R^d )$ and $C(\TT^- ,\R^d )$ respectively. The discrete antiderivative is denoted by $J_{\Delta}$. \\

We define the set $L^2_{\TT} (\R^d )$ to be the discrete functions $F\in C (\TT ,\R^d )$ such that $J_{\Delta} (F\star F )$ is well defined where $\star$ denotes the discrete product over $C(\TT ,\R^d )$ defined in (\cite{cp1},$\S$. 3.3.1). \\

We have the following results (see \cite{cp1}, $\S$.5.3.2 and $\S$.6.2) : 

\begin{theorem}[Discrete integration by part formula] 
Let $F\in C(\TT ,\R^d)$ and $G\in C_0 (\TT ,\R^d )$, we have 
\begin{equation}
\left .
\begin{array}{lll}
\left [ J_{\Delta}(F\star \Delta (G)) \right ]_N & = & - \left [ J_{\Delta} (F \star (\nabla G)) \right ]_N , \\
\left [ J_{\nabla}(F\star \nabla (G)) \right ]_N & = & - \left [ J_{\nabla} (F \star (\Delta G)) \right ]_N .
\end{array}
\right .
\end{equation}
\end{theorem}

\begin{lemma}
Let $F\in C(\TT,\R^d )$ such $\left [ J_{\Delta}(F\star G) \right ]_N =0$ for all $G\in C_0(\TT,\R^d )$ then $F_{k}=0$ for $k=1,\ N-1$.
\end{lemma}

\part{Discrete Hamiltonian systems}

Discrete Hamiltonian systems have already been defined by many authors as for example in (\cite{marsden,lall,lubisch}). However, these formulations do not follow the strategy of discrete embedding which is aimed as providing an explicit relation between algebraic and analytic continuous structures and there discrete analogue (see \cite{cp1}). As a consequence, we provide a self-contain derivation of discrete Hamiltonian systems in this setting. A comparison with previous approaches is given.

\section{Reminder about Hamiltonian systems}

\subsection{Hamiltonian systems}

For simplicity we consider time independant Hamiltonian and Lagrangian.

\begin{definition}[Classical Hamiltonian]
A classical Hamiltonian is a function $H : \R^d \times \R^d \rightarrow \R$ such that for $(q,p)\in C^1([a,b],\R^d) \times C^1([a,b],\R^d)$ we have the time evolution of $(q,p)$ given by the classical Hamilton's equations
\begin{equation}
\left\{
\begin{array}{l l}
\dot{q}&=\frac{\partial H(q,p)}{\partial p} \\
\dot{p}&=-\frac{\partial H(q,p)}{\partial q}
\end{array}
\right.
\label{def_hamilton}
\end{equation}
\end{definition}

A vectorial notation is obtained posing $z=(q,p)^\mathsf{T}$ and $\nabla H = (\frac{\partial H}{\partial q} , \frac{\partial H}{\partial p} )^\mathsf{T}$ where $\mathsf{T}$ denotes the transposition. The Hamilton's equations are then written as
\begin{equation}
\frac{dz}{dt} = J \cdot \nabla H,
\end{equation}
where $J = \begin{pmatrix} 0 & I_d \\ -I_d & 0 \end{pmatrix}$ with $I_d$ the identity matrix on $\R^d$ denotes the symplectic matrix.\\

We also denote by $X_H$ the associated vector field defined by 
\begin{equation}
X_H  =  \frac{\partial H}{\partial q} \partial_q -\frac{\partial H}{\partial p} \partial_p .
\end{equation} 

An important property of Hamiltonian systems is that there solutions correspond to \emph{critical points} of a given functional, i.e. follow from a \emph{variational principle}.

\begin{theorem}
The points $(q,p)\in C^1([a,b],\R^d) \times C^1([a,b],\R^d)$ satisfying Hamilton's equations are critical points of the functional
\begin{equation}
\fonction{\mathcal{L}_H}{C^1([a,b],\R^d)\times C^1([a,b],\R^d)}{\R}{(q,p)}{\mathcal{L}_H(q,p) = \di\int_{a}^{b} \di L_H (q(t),p(t),\dot{q}(t),\dot{p}(t))}
\end{equation}
where $\fonctionsansdef{L_H}{\R^n \times \R^n \times \R^n \times \R^n}{\R}$ is the Lagrangian defined by
\begin{equation}
\di L_H(x,y,v,w)=y\cdot v-H(x,y).
\end{equation}
\label{thm_hamiltonian}
\end{theorem}

\subsection{Lagrangian versus Hamiltonian}

As we have already a consistent discrete theory of Lagrangian system, we will used this derivation of Hamiltonian systems in order to define \emph{discrete Hamiltonian systems}. Doing so, we will see that all the objects are related by the discrete embedding procedure.\\

Let $L$ be a Lagrangian and denote by \mbox{EL} the corresponding Euler-Lagrange equation given by 
\begin{equation}
\di \frac{d}{dt} 
\left [ 
\frac{\partial L}{\partial v} (q,\dot{q} ,t ) 
\right ] 
=\di\frac{\partial L}{\partial q} (q,\dot{q} ,t) .
\label{EL}
\end{equation}
We assume that the Lagrangian is \emph{admissible}, i.e. that the map $v\rightarrow \di\frac{\partial L}{\partial v}$ is invertible for all $(q,t) \in \R^n \times \R$. As a consequence, we can introduce the \emph{moment} variable
\begin{equation}
p=\frac{\partial L}{\partial v} (q,\dot{q},t) ,
\end{equation}
in order to rewrite the Euler-Lagrange equation as a first order system of differential equations given by
\begin{equation}
\left .
\begin{array}{lll}
\dot{q} & = & g(q,p,t) ,\\
\dot{p} & = & \di\frac{\partial L}{\partial q} (q, g(q,p,t) ,t) ,
\end{array}
\right .
\label{change}
\end{equation}
where $g$ is the inverse of $\di\frac{\partial L}{\partial v}$. This change of variable will be called \emph{Legendre transform} in the following. Introducing the Hamiltonian function 
\begin{equation}
\label{ham}
H(q,p,t) = L(q,g(q,p,t),t) -pg(q,p,t) ,
\end{equation}
one proves that equation (\ref{change}) is Hamiltonian with respect to $H$.

\section{Discrete Hamiltonian systems}

There exists many way to define discrete analogue of Hamiltonian systems. In this Section, we quickly remind some classical results which can be found for example in (\cite{marsden},\cite{lubisch},\cite{lall}) in a slightly different form. 

We follow the {\it discrete embedding} formalism exposed in \cite{cp1}. Explicit comparisons with other formulations are given. In order to obtain a coherent definition of discrete Hamiltonian systems, we follow the usual continuous derivation, starting from a Lagrangian system to its Hamiltonian formulation via the Legendre transform. As we will see, all these structures can be obtained via a direct discrete embedding of the corresponding continuous structures. This property is exactly what is missed in the existing literature. 

\subsection{Discrete Lagrangian systems}

The discrete analogue of the Euler-Lagrange equation is given by (see \cite{cp1}) :
\begin{equation}
\nabla \left [ \di\frac{\partial L}{\partial v} (Q,\Delta Q ,T) \right ] =\di\frac{\partial L}{\partial q} (Q,\Delta Q ,T ).
\end{equation}
We introduce the set of variables 
\begin{equation}
P =\di \frac{\partial L}{\partial v} (Q,\Delta Q ,T) ,
\end{equation}
which are the classical \emph{moment variables} in classical mechanics. Thanks to the admissibility of $L$, we have $\Delta Q = g(P,Q,T)$. Using these variables, the discrete Euler-Lagrange equation is then equivalent to the discrete system 
\begin{equation}
\left .
\begin{array}{lll}
P & = & \di \frac{\partial L}{\partial v} (Q, g(P,Q,T) ,T) ,\\
\nabla P & = &  \di\frac{\partial L}{\partial q} (Q, g(P,Q,T) ,T ).
\end{array}
\right .
\end{equation}
The form of this discrete system is obtain from the continuous case taking care of the duality between the variable $Q$ and $P$ inducing a change between $\Delta$ and $\nabla$ for the discretisation of the derivative.\\

Introducing the \emph{Hamiltonian function} 
\begin{equation}
H(Q,P,T)=L(Q,g(P,Q,T),T) - P g(P,Q,T ) 
\end{equation}
which is exactly the \emph{discrete} form of the continuous Hamiltonian $H$ associated to $L$. 

\begin{remark}
It must be noted that the previous result holds independently of the discrete embedding which is chosen. As a consequence, we have always the \emph{same} Hamiltonian function. This is completely different from the result of Lall-West \cite{lall} where different Hamiltonian are introduced. 
\end{remark}

Using the Hamiltonian, one obtain the Hamiltonian form of the Euler-Lagrange equation :
\begin{equation}
\left .
\begin{array}{lll}
\Delta Q & = & \di\frac{\partial H}{\partial P} ,\\
\nabla P & = & -\di\frac{\partial H}{\partial Q} .
\end{array}
\right .
\end{equation}

\subsection{Discrete Hamiltonian system}

The previous derivation of a \emph{discrete} form of Hamiltonian systems using the discrete Euler-Lagrange equations leads to the following definition :
 
\begin{definition}[Discrete Hamiltonian]
A discrete Hamiltonian is a function $H : \R^d \times \R^d \rightarrow \R$ such that for $(Q,P)\in C(\TT,\R^d) \times C(\TT,\R^d)$ we have the time evolution of $(Q,P)$ given by the discrete $\Delta-$Hamilton equations (resp. $\nabla-$Hamilton equations)
\begin{align}
\left\{
\begin{array}{l l}
(\Delta Q) &=\frac{\partial H}{\partial P} ,\\
(\nabla P) &=-\frac{\partial H}{\partial Q} ,
\end{array} 
\right.
\label{def_ddisc_hamiltonian}
\end{align}
over $\TT^{\pm}$.
\end{definition}

Aa a convenient notation, in analogy with the continuous case, we introduce what we call a \emph{discrete vector field} associated to the difference equation (\ref{def_ddisc_hamiltonian}) denoted by $X_{\Delta ,\nabla ,H}$ and \emph{formally} denoted by
\begin{equation}
X_{\Delta ,\nabla, H} = \frac{\partial H}{\partial P} \partial_{\Delta ,Q} -\frac{\partial H}{\partial Q} 
\partial_{\nabla ,P} .
\end{equation}

With this notation, we have the following \emph{commutative} diagram indicating the coherence of our construction with respect to the continuous one :\\

The discrete diagram
\[\xymatrixcolsep{5pc} \xymatrixrowsep{5pc}
\xymatrix{ 
{L_h}  \ar[d]^-{\mbox{dlap}}\ar[r]^-{disc } & {H_h} \ar[d]^-{\mbox{def}} \\ 
{EL_h} \ar[r]^-{Legendre} & {X_{\Delta ,\nabla ,H}}
}
\]
where \emph{d.l.a.p} denotes the discrete least action principle, is exactly the same as the continuous one
\[\xymatrixcolsep{5pc} \xymatrixrowsep{5pc}
\xymatrix{ 
{L}  \ar[d]^-{\mbox{lap}}\ar[r]^-{disc } & {H} \ar[d]^-{\mbox{def}} \\ 
{EL} \ar[r]^-{Legendre} & {X_H}
}
\]
where \emph{l.a.p} denotes the least action principle. All the objects are related via the \emph{discrete $\Delta$-embedding} as given by
\[\xymatrixcolsep{5pc} \xymatrixrowsep{5pc}
\xymatrix{ 
{L}  \ar[d]^-{\mbox{disc}}\ar[r]^-{def } & {H} \ar[d]^-{\mbox{disc}} \\ 
{L_h} \ar[r]^-{def} & {H_h}
}
\]

It must be noted however that the previous definition of \emph{discrete Hamiltonian system} does not coincide with the \emph{discrete differential embedding} of the continuous Hamiltonian system which gives
\begin{equation}
\left\{
\begin{array}{l l}
(\Delta Q) &=\di\frac{\partial H}{\partial P} ,\\
(\Delta P) &=-\di\frac{\partial H}{\partial Q} ,
\end{array} 
\right.
\end{equation}
over $\TT^{\pm}$.
We denote by $X_{\Delta ,H}$ the associated discrete vector field
\begin{equation}
X_{\Delta ,H} = \frac{\partial H}{\partial P} \partial_{\Delta ,Q} -\frac{\partial H}{\partial Q} 
\partial_{\Delta ,P} .
\end{equation}

We will see that this phenomenon is responsible for the \emph{non-coherence} between the differential and variational discrete embedding of Hamiltonian systems. 

\subsection{Discrete Hamilton principle}

A main property of discrete Hamiltonian systems is that they are preserving an essential feature of continuous Hamiltonian systems, i.e. the variational structure. Precisely, let $\mathcal{L}_{\Delta,H} (Q,P)$ be the discrete functional obtained by the $\Delta$-discrete embedding of the continuous one and defined by
\begin{equation}
\mathcal{L}_{\Delta,H} (Q,P) = \left [ J_{\Delta} ( P\star \Delta Q - H(Q,P) ) \right ]_N .
\end{equation}

We have :

\begin{theorem}
\label{discrete-variational-hamiltonian}
A couple $(Q,P)$ is a critical point of $\mathcal{L}_{\Delta,H} (Q,P)$ if and only if it satisfies the discrete Hamiltonian system associated to $H$.
\end{theorem}

The previous result complete the picture of our construction leading to a full commutative diagram between all the object  and given by

\[\xymatrixcolsep{5pc} \xymatrixrowsep{5pc}
\xymatrix{ {\mathcal{L}}  \ar[d]^-{\mbox{lap}}\ar[r]^-{disc } & {\mathcal{L}_{\Delta ,h}} \ar[d]^-{\mbox{discrete l.a.p}} \\ {X_H} \ar[r]^-{def} & {X_{\Delta ,\nabla,H}}}
\]

However, the discrete embedding formalism can be used to obtain a different definition of Hamiltonian system using the differential discrete $\Delta$-embedding. In this case, we obtain a non-commutative diagram :

\[\xymatrixcolsep{5pc} \xymatrixrowsep{5pc}
\xymatrix{ {\mathcal{L}}  \ar[d]^-{\mbox{lap}}\ar[r]^-{disc } & {\mathcal{L}_{\Delta ,h}} \ar[d]^-{\mbox{discrete l.a.p}} \\ {X_H} \ar[r]^-{disc} & {X_{\Delta ,H} \not= X_{\Delta ,\nabla,H}}}
\]
which is reminiscent of the non-coherence of the procedure between differential and variational discrete embedding already in the Lagrangian case (see \cite{cp1}).

\begin{proof}
The Fredchet derivative of the discrete functional $\mathcal{L}_{\Delta,H}$ is given for all $U,V\in C_0(\TT,\R)$ by
\begin{equation}
\mathcal{D}\mathcal{L}_{\Delta,H}(Q,P)(U,V)=\left [ J_{\Delta}\left(P\star\Delta U + V\star \Delta Q - \frac{\partial H(Q,P)}{\partial Q}\star U - \frac{\partial H(Q,P)}{\partial P}\star V\right) \right ]_N .
\end{equation}
As $U\in C_0(\TT,\R)$, we obtain using the \emph{discrete integration by parts formula} :
\begin{equation}
\mathcal{D}\mathcal{L}_{\Delta,H}(Q,P)(U,V)=
\left [ 
J_{\Delta}
\left( 
-\nabla P\star U + V\star \Delta Q - \frac{\partial H(Q,P)}{\partial Q}\star U - \frac{\partial H(Q,P)}{\partial P}\star V
\right)
\right ]_N
.
\end{equation}
which can be rewritten as
\begin{equation}
\mathcal{D}\mathcal{L}_{\Delta,H}(Q,P)(U,V)=
\left [ 
J_{\nabla}
\left( 
V\star \left(\Delta Q - \frac{\partial H(Q,P)}{\partial P}\right) - U\star\left( \nabla P + \frac{\partial H(Q,P)}{\partial Q}\right)
\right)
\right ]_N
.
\end{equation}
By definition, a critical point of $\mathcal{L}_{\Delta,H}$ satisfies $\mathcal{D}\mathcal{L}_{\Delta,H}(Q,P)(U,V)=0$ for all $U,V\in C_0(\TT,\R)$. Using the \emph{discrete Dubois-Raymond lemma} we deduce that
\begin{equation}
\left\{
\begin{array}{lll}
(\Delta Q) & = & \frac{\partial H}{\partial Q} , \\
(\nabla P) & = & -\frac{\partial H}{\partial P} ,
\end{array} 
\right .
\end{equation}
over $\TT^{\pm}$.
\end{proof}

\subsection{Comparison with West and al. and Opris and al. approaches}

\subsubsection{West and al. approach}

In \cite{marsden} and \cite{lall}, the authors introduce a definition of \emph{discrete Hamiltonian systems}. The second paper uses the discrete Lagrangian formalism developed in \cite{marsden}. 

The main difference between these approaches and our result is that our definitions are coherent with the formalism of \emph{discrete embedding} developed in \cite{cp1}. This remark has strong consequences. In particular, in the discrete embedding formalism there is no changes between the Hamiltonian or Lagrangian defining the continuous system and the discrete analogue. The discretisation preserves all the relations between the different objects (Lagrangian, Hamiltonian, Legendre transform) globally leading to the previous commutative diagrams. In the contrary, in \cite{marsden} or \cite{lall}, the discrete Lagrangian associated to a continuous one change the Lagrangian function (see \cite{marsden},p.363 or \cite{lubisch}, p.192-195). Moreover, it breaks the differential form of the Euler-Lagrange equation (see $\S$.8.3.5 in \cite{cp1}). 

The same is true in the Hamiltonian case. The authors introduce discrete Hamiltonian functions which do not coincide with the discrete functions associated with the continuous Hamiltonian function (see \cite{lall},p.5512-5513 and \cite{cp1}, $\S$.3.1). As a consequence, they are able to show an \emph{analogy} between the continuous and the discrete objects (see \cite{marsden},Part I,$\S$.1.6) but miss the fact that this is more than an analogy as proved by the previous diagrams. Using discrete embeddings (more precisely the variational one) all the structures and objects are exactly transported in the discrete framework. 

\subsubsection{Opris and al. approach}
\label{compardef}
In (\cite{opri2},$\S$.5), the authors give a definition of discrete Hamiltonian systems. This definition is made independently of any discretisation procedure. As a consequence, an explicit comparison of their approach is difficult. However, one can find some similarities in the definitions even if the global picture given by our framework is different. Indeed, the starting point is a discrete Legendre transform with respect to a discrete Lagrangian. As a consequence, they obtain a Hamiltonian function which has the same form as in the continuous case (see \cite{opri2}, formula (27) p.26) and a form of the discrete Hamiltonian equations which are similar to ours (see \cite{opri2}, formula (32) p.26). However, due to the framework used by the authors, the comparison to the continuous case is very difficult contrary to our derivation of discrete Hamiltonian systems.

\part{Discrete Helmholtz conditions}

In this Section, we solve the inverse problem of the discrete calculus of variations in the Hamiltonian case. We first recall the usual way to derive the Helmholtz conditions following the presentation made by Santilli \cite{santilli}. We have two main derivations : 
\begin{itemize}
\item One is related to the characterization of Hamiltonian systems via the \emph{symplectic two-differential form} and the fact that by duality the associated one-differential form to a Hamiltonian vector field is closed. Such conditions are called \emph{integrability conditions}.

\item The second one use the characterization of Hamiltonian systems via the self adjointness of the Frechet derivative associated to the differential operator associated to the equation. These conditions are usually called \emph{Helmholtz conditions}.
\end{itemize}
Of course, we have coincidence of the two procedures in the usual case. In the discrete case however, it seems that the second one is more appropriate to a generalization, avoiding the definition of \emph{discrete differential forms} as done for example by E.L. Mansfeld and P.E. Hydon (\cite{hydon1},\cite{hydon2}) or Z. Bartosiewicz and al. \cite{barto} in the time-scale setting. As a consequence, we follow the second way to obtain the discrete analogue of the Helmholtz conditions.

\section{Hemlholtz conditions for Hamiltonian systems}

\subsection{Symplectic scalar product}

In this Section we work on $\R^{2d},d \ge 1, d \in \N$. The \emph{symplectic scalar product} $\langle \cdot,\cdot \rangle_J$ is defined for all $X,Y \in \R^{2d}$ by
\begin{equation}
\langle X,Y\rangle_J = \langle X,JY\rangle ,
\end{equation}
where $\langle \cdot,\cdot \rangle$ denotes the usual scalar product and $J = \begin{pmatrix} 0 & I_d \\ -I_d & 0 \end{pmatrix}$ with $I_d$ the identity matrix on $\R^d$. 

We also consider the $L^2$ symplectic scalar product induced by $\langle \cdot,\cdot \rangle_J$ defined for $f,g \in C^1([a,b],\R^{2d})$ by 
\begin{equation}
\langle f,g \rangle_{L^2,J}=\di\int_{a}^{b} \langle f(t),g(t) \rangle_Jdt \ .
\end{equation}

\subsection{Adjoin of a differential operator}

In the following, we consider first order differential equations of the form
\begin{equation}
\label{equagen}
\frac{d}{dt}\begin{pmatrix}q \\ p \end{pmatrix} = \begin{pmatrix} X_q(q,p) \\ X_p(q,p) \end{pmatrix}.
\end{equation}
The associated differential operator is written as
\begin{equation}
O^{a,b}_X(q,p) = \begin{pmatrix} \dot{q} - X_q(q,p) \\ \dot{p} - X_p(q,p) \end{pmatrix} \ . 
\label{operatorO}
\end{equation}

A \emph{natural} notion of adjoin for a differential operator is then defined.

\begin{definition}
Let $\fonctionsansdef{A}{C^1([a,b],\R^{2n})}{C^1([a,b],\R^{2n})}$. We define the adjoin $A^*_J$ of $A$ with respect to $<\cdot,\cdot>_{L^2,J}$ by
\begin{equation}
<A \cdot f, g>_{L^2,J} = <A^{*}_J \cdot g, f>_{L^2,J} \ .
\end{equation}
\end{definition}

An operator $A$ will be called \emph{self-adjoin} if $A=A^{*}_J$ with respect to the $L^2$ symplectic scalar product.

\subsection{Hamiltonian Helmholtz conditions}

The Helmholtz's conditions in the Hamiltonian case are given by (see \cite{santilli}, Theorem. 3.12.1, p.176-177) :

\begin{theorem}[Hamiltonian Helmholtz conditions]
Let $X(q,p)^\mathsf{T}= (X_q(q,p) , X_p(q,p) )$ be a vector field. The differential equation (\ref{equagen}) is Hamiltonian if and only the associated differential operator $O^{a,b}_X$ given by (\ref{operatorO}) has a  self adjoin Frechet derivative with respect to the symplectic scalar product. 

In this case the Hamiltonian is given by 
\begin{equation}
H(q,p)=\int_{0}^{1}\left[p \cdot X_q(\lambda q, \lambda p) - q\cdot X_p(\lambda q, \lambda p) \right]d\lambda
\end{equation}
\end{theorem}

The conditions for the adjointness of the differential operator can be made \emph{explicit}. They coincide with the \emph{integrability conditions} characterizing the exactness of the one-form associated to the vector field by duality (see \cite{santilli}, Thm.2.7.3 p.88).

\begin{theorem}[Integrability conditions]
Let $X(q,p)^\mathsf{T}= (X_q(q,p) , X_p(q,p) )$ be a vector field. The differential operator $O^{a,b}_X$ given by (\ref{operatorO}) has a  self adjoin Frechet derivative with respect to the symplectic scalar product if and only if 
\begin{equation}
\frac{\partial X_q}{\partial q} + \left(\frac{\partial X_p}{\partial p} \right)^\mathsf{T} = 0, \quad \frac{\partial X_q}{\partial p} \ \text{and} \ \frac{\partial X_p}{\partial q} \ \text{are symmetric} .
\end{equation}
\end{theorem}

Of course, the first condition corresponds to the fact that Hamiltonian systems are \emph{divergence free}, i.e. that we have $\mbox{div} X =0$. 

\section{Discrete Helmholtz's conditions}

We derive the main result of the paper giving the characterization of first order difference systems which ate discrete Hamiltonian systems. 

\subsection{Discrete symplectic scalar product}

Following the continuous case, we introduce a discrete analogue of a symplectic scalar product and the notion of self-adjointness for discrete difference operators.\\

We recall that for $X,Y \in L^2_{\TT}$ the scalar product is given by for all $X,Y\in L^2_{\TT}$ by
\begin{equation}
\langle X,Y \rangle_{L^2 ,\Delta}  = \left [ J_{\Delta} \left ( X \star Y \right ) \right ]_N .
\end{equation}

The symplectic version is obtained as in the continuous case.

\begin{definition}
Let $X,Y \in C(\TT,\R^{2d})$. The symplectic $\Delta$-scalar product over $L^2_{\TT}$ is defined by
\begin{equation}
\langle X,Y \rangle_{L^2_{\TT} ,\Delta, J} = \langle X,J\cdot Y\rangle_{L^2_{\TT},\Delta }.
\end{equation}
\end{definition}

\subsection{Adjoin of a discrete finite difference operator}

We consider first order difference equations of the form
\begin{equation}
\label{equagend}
\begin{pmatrix} \Delta Q \\ \nabla P \end{pmatrix} = \begin{pmatrix} X_Q( Q,P) \\ X_P(Q,P) \end{pmatrix}.
\end{equation}
The associated finite difference operator is written as
\begin{equation}
O^{\TT}_X(Q,P) = \begin{pmatrix} \Delta{Q} - X_Q(Q,P) \\ \nabla{P} - X_P(Q,P) \end{pmatrix} \ . 
\label{operatorOd}
\end{equation}

A notion of {\it symplectic adjoin} for discrete finite differences operators can be defined :

\begin{definition}
Let $\fonctionsansdef{A}{C(\TT,\R^{2d})}{C(\TT,\R^{2d})}$ be a discrete difference operator. We define the adjoin $A^*_J$ of $A$ with respect to $\langle \cdot,\cdot \rangle_{L^2_{\TT},\Delta ,J}$ by
\begin{equation}
\langle A(X),Y \rangle_{L^2_{\TT},\Delta ,J} = \langle A^{*}_J (Y), X \rangle_{L^2_{\TT},\Delta ,J}.
\end{equation}
\label{adjointdiscret}
\end{definition}

As in the continuous case, a main role will be played by self-adjoin discrete operators.

\begin{definition}
A discrete operator $A$ is said to be self-adjoin with respect to symplectic $\Delta$-scalar product over $L^2_{\TT}$ if $A=A^{*}_J $.
\label{selfadjointdiscret}
\end{definition}

\subsection{Discrete Helmholtz conditions}

Let $X(Q,P)^\mathsf{T}= (X_Q(Q,P) , X_P(Q,P)) \label{disc_operator_dn}$ be a discrete vector field. The main result of this Section is the discrete analogue of the Hamiltonian Helmholtz conditions for discrete difference equations.

\begin{theorem}[Discrete Hamiltonian Helmholtz conditions]
\label{main1}
Let $X(Q,P)^\mathsf{T}= (X_Q(Q,P) , X_P(Q,P) )$ be a discrete vector field. The discrete difference equation (\ref{equagend}) is a discrete Hamiltonian equation if and only if the operator $O^{\TT}_X$ defined by (\ref{operatorOd}) has a self-adjoin Frechet derivative with respect to the discrete $L^2_{\TT}$ symplectic $\Delta$-scalar product. 

Moreover, in this case the Hamiltonian is given by 
\begin{equation}
\label{ham_form}
H(Q,P)=\int_{0}^{1}\left[P \star X_Q(\lambda Q, \lambda P) - Q\star X_P(\lambda Q, \lambda P) \right]d\lambda .
\end{equation}
\end{theorem}

The proof is given in Section \ref{proofmain1}.\\

We give an explicit characterization of discrete vector fields satisfying the Hamiltonian Helmholtz conditions. By coherence with the continuous case, we call them \emph{discrete integrability conditions}.

\begin{proposition}[Discrete integrability conditions]
The operator $O^{\TT}_X$ defined by (\ref{operatorOd}) has a self-adjoin Frechet derivative at $(Q,P)\in C(\TT,\R)\times C(\TT,\R)$ if and only if the conditions
\begin{align}
&\frac{\partial X_Q}{\partial Q} + \left(\frac{\partial X_P}{\partial P} \right)^\mathsf{T} = 0 \quad \text{(CH1)} \ , \\ 
&\frac{\partial X_Q}{\partial P} \ \text{and} \ \frac{\partial X_P}{\partial Q} \ \text{are symmetric} \quad \text{(CH2)},
\end{align}
are satisfied over $\TT^\pm$.
\end{proposition}

The proof easily follows from the following Proposition. 

\begin{proposition}
\label{main2}
Let $U,V \in C_0(\TT,\R^{d})$. The Frechet derivative $DO(Q,P)$ of (\ref{operatorOd}) is given by
\begin{align}
DO(Q,P)(U,V) = \begin{pmatrix} \Delta U -\frac{\partial X_Q}{\partial Q} \star U -\frac{\partial X_Q }{\partial P}\star V \\ \nabla V -\frac{\partial X_P}{\partial Q} \star U -\frac{\partial X_P}{\partial P}\star V \end{pmatrix}
\end{align}
and his adjoin ${DO^*_J}(Q,P)$ with respect to the symplectic $\Delta$-scalar product is given by
\begin{align}
{DO_J^*}(Q,P)(U,V) = \begin{pmatrix} \Delta U +\left(\frac{\partial X_P}{\partial P}\right)^\mathsf{T}\star U -\left(\frac{\partial X_Q}{\partial P}\right)^\mathsf{T} \star V \\ \nabla V -\left(\frac{\partial X_P}{\partial Q}\right)^\mathsf{T}\star U +\left(\frac{\partial X_Q}{\partial Q}\right)^\mathsf{T} \star V\end{pmatrix}.
\end{align}
\end{proposition}

\begin{proof}
Let $U,V \in C_0(\TT,\R^{d})$ and $A,B \in C(\TT,\R^{d})$. The Frechet derivative $DO(Q,P)$ follows from simple computations and is given by 
\begin{align}
\langle DO(Q,P)(U,V),(A,B) \rangle_{L^2_{\TT}, \Delta , J}=
\left [ 
J_{\Delta}\bigg(  \right . & \Delta U \star B - \left(\frac{\partial X_Q}{\partial Q}\star U \right)\star B-\left(\frac{\partial X_Q}{\partial P}\star V\right)\star B \nonumber \\
 &- \left . 
 \nabla V \star A + \left(\frac{\partial X_P}{\partial Q}\star U \right)\star A + \left(\frac{\partial X_P}{\partial P}\star V\right)\star A \bigg) 
 \right ]_N .\nonumber 
\end{align}
Using the discrete integration by part formula, we obtain 
\begin{align}
\langle DO(Q,P)(U,V),(A,B) \rangle_{L^2_{\TT},\Delta ,J}= \left [ 
J_{\Delta}\Bigg( \right . &- U \star \left[\nabla B -\left(\frac{\partial X_P}{\partial Q}\right)^\mathsf{T}\star A +\left(\frac{\partial X_Q}{\partial Q}\right)^\mathsf{T} \star B  \right] \nonumber \\
&+\left . V \star \left[ \Delta A +\left(\frac{\partial X_P}{\partial P}\right)^\mathsf{T}\star A -\left(\frac{\partial X_Q}{\partial P}\right)^\mathsf{T} \star B \right] \Bigg) \right ]_N . \nonumber 
\end{align}
As a consequence, the symplectic discrete adjoint of $DO$ is given by
\begin{equation}
{DO_J^*}(Q,P)(A,B) = \begin{pmatrix} \Delta A +\left(\frac{\partial X_P}{\partial P}\right)^\mathsf{T}\star A -\left(\frac{\partial X_Q}{\partial P}\right)^\mathsf{T} \star B \\ \nabla B -\left(\frac{\partial X_P}{\partial Q}\right)^\mathsf{T}\star A +\left(\frac{\partial X_Q}{\partial Q}\right)^\mathsf{T} \star B\end{pmatrix}
\end{equation}
which concludes the proof.
\end{proof}

\subsection{Proof of Theorem \ref{main1}}
\label{proofmain1}

If $X$ is a discrete Hamiltonian vector field then there exist a function $H \in C(\R^d \times \R^d , \R )$ such that $X_Q=\frac{\partial H}{\partial P}$ and $X_P=-\frac{\partial H}{\partial Q}$. The discrete Hamiltonian Helmholtz conditions of order one are satisfied. Indeed, we have $\frac{\partial^2 H}{\partial P \partial Q} = \frac{\partial^2 H}{\partial Q \partial P}$ by the Schwarz lemma. \\

Reciprocally, we suppose that $X$ satisfies the discrete Hamiltonian Helmholtz conditions. We define the function $H$ as 
\begin{equation}
H(Q,P)=\int_{0}^{1}\left[ P \star X_Q(\lambda Q, \lambda P) - Q\star X_P(\lambda Q, \lambda P) \right]d\lambda
\end{equation}
with the functional 
\begin{equation}
\mathcal{L}_{\Delta,H} (Q,P) = \left [ J_{\Delta} ( P \star \Delta Q - H(Q,P) ) \right ]_N .
\end{equation} 
The Frechet derivative of $\mathcal{L}_{\Delta,H}$ at $(Q,P)$ along $U,V\in C_0(\TT,\R)$ is given by 
\begin{equation}
\mathcal{D}\mathcal{L}_{\Delta,H}(Q,P)(U,V)=
\left [
J_{\Delta}\left(P\star\Delta U + V\star \Delta Q - DH(Q,P)(U,V)\right)
\right ] .
\end{equation}
Using the discrete integration by part formula, we obtain
\begin{equation}
\mathcal{D}\mathcal{L}_{\Delta,H}(Q,P)(U,V)= \left [ J_{\Delta}\left(-\nabla P\star U +  V\star \Delta Q - DH(Q,P)(U,V)\right) 
\right ]_N .
\end{equation}
We compute $DH(Q,P)(U,V)$ over $\TT^\pm$. We denote $(\ast)=(\lambda Q, \lambda P)$. We have
\begin{align}
DH(Q,P)(U,V) &= \int_{0}^{1} \bigg( V\star X_Q (\ast) -U\star X_P(\ast) + \lambda P\star\left[ \frac{\partial X_Q(\ast)}{\partial Q}\star U + \frac{\partial X_Q(\ast)}{\partial P}\star V\right] \nonumber \\
& - \lambda Q\star\left[ \frac{\partial X_P(\ast)}{\partial Q}\star U + \frac{\partial X_P(\ast)}{\partial P}\star V\right] \bigg)d\lambda \nonumber \\
&= \int_{0}^{1} \bigg( V\star X_Q (\ast) -U\star X_P(\ast) + \lambda U\star\left[ \left(\frac{\partial X_Q(\ast)}{\partial Q}\right)^\mathsf{T} \star P - \left(\frac{\partial X_P(\ast)}{\partial Q}\right)^\mathsf{T} \star Q \right] \nonumber \\
& + \lambda V\star\left[ \left(\frac{\partial X_Q(\ast)}{\partial P}\right)^\mathsf{T} \star P - \left(\frac{\partial X_P(\ast)}{\partial P}\right)^\mathsf{T} \star Q\right] \bigg)d\lambda \nonumber.
\end{align}

Using the Hamiltonian Helmholtz conditions we obtain 
\begin{align}
DH(Q,P)(U,V) & =\int_{0}^{1} \bigg( V\star X_Q (\ast) -U\star X_P(\ast) - \lambda U\star\left[ \frac{\partial X_P(\ast)}{\partial P} \star P + \frac{\partial X_P(\ast)}{\partial Q}\star Q \right] \nonumber \\
& + \lambda V\star\left[ \frac{\partial X_Q(\ast)}{\partial P} \star P + \frac{\partial X_Q(\ast)}{\partial Q}\star Q\right] \bigg)d\lambda \nonumber.
\end{align}

Remarking that

\begin{align}
\frac{\partial X_P(\ast)}{\partial \lambda} &= \frac{\partial X_P(\ast)}{\partial P} \star P + \frac{\partial X_P(\ast)}{\partial Q}\star Q \ , \\
\frac{\partial X_Q(\ast)}{\partial \lambda} &= \frac{\partial X_Q(\ast)}{\partial P} \star P + \frac{\partial X_Q(\ast)}{\partial Q}\star Q \\
\end{align}
we deduce that
\begin{align}
DH(Q,P)(U,V) &= \int_{0}^{1} \frac{\partial}{\partial\lambda}\left( \lambda V\star X_Q(\ast) - \lambda U\star X_P(\ast)\right) d\lambda ,
\end{align}
and finally, integrating with respect to $\lambda$ :
\begin{align}
DH(Q,P)(U,V) &= V\star X_Q(Q,P) - U\star X_P(Q,P).
\end{align} 
Replacing $DH(Q,P)(U,V)$ by its expression in the differential of $\mathcal{L}_{\Delta,H}$, we obtain
\begin{align}
\mathcal{D}\mathcal{L}_{\Delta,H}(Q,P)(U,V)= \left [ J_{\Delta }\left( V\star \left(\Delta Q - X_Q(Q,P)\right) - U\star\left( \nabla P -X_P(Q,P)\right)\right) \right ]_N .
\end{align}
A critical point $(Q,P)$ satisfies $\mathcal{D}\mathcal{L}_{\Delta,H}(Q,P)(U,V)=0$ for all $U,V \in C_0 (\TT ,\R^n)$. Using the discrete Dubois-Raymond lemma we deduce that
\begin{align}
\left\{
\begin{array}{l l}
(\Delta Q) &= X_Q  ,\\
(\nabla P) &=X_P ,
\end{array}
\right.
\end{align}
over $\TT^\pm$. This concludes the proof.

\subsection{Comparison with Opris and al. results}

In \cite{opri1,opri2}, the authors prove discrete analogues of the Helmholtz's theorem. 

The result in \cite{opri1} is closer to our previous work \cite{cresson-bourdin} as they consider discrete Lagrangian systems.

In \cite{opri2}, the authors define first discrete Hamiltonian systems and discuss the Helmholtz conditions for first order difference equations (see \cite{opri2},$\S$. 6). As already noted in $\S$.\ref{compardef}, it is not clear to compare their results with the continuous case. Indeed, the objects are not related to discretisation of the continuous one. As a consequence, the problem is in one hand more general but in the other one easier to solve. Indeed, no constraints are associated with an underlying continuous structure to respect. We can remark that they use also the self-adjoin characterization of the Hamiltonian structure (see \cite{opri2},Prop.9) and not a characterization using discrete analogue of one-form and closeness.

\part{Applications}

In this Section, we study several discrete finite differences systems for which we decide it they admit a discrete Hamiltonian structure.

\section{The linear case}

Let us consider the following discrete linear system defined over $\TT^\pm$ by
\begin{equation}
\left\{
\begin{array}{l l}
\Delta{Q}&=\alpha Q + \beta P, \\
\nabla{P}&=\gamma Q + \delta P .
\end{array}
\right.
\label{ex1_eq}
\end{equation}
where $\alpha$, $\beta$, $\gamma$ and $\delta$ are constants and $Q,P\in C(\TT,\R^n)$. The Helmholtz condition (CH2) is clearly satisfied. However, the system \eqref{ex1_eq} satisfies the condition (CH1) if and only if $\alpha+\delta=0$.

As a consequence, Hamiltonian linear difference equations are of the form
\begin{equation}
\left\{
\begin{array}{l l}
\Delta{Q}&=\alpha Q + \beta P, \\
\nabla{P}&=\gamma Q -\alpha P,
\end{array}
\right.
\end{equation}
Using formula \eqref{ham_form} we compute explicitly the Hamiltonian which is given by
\begin{equation}
H(Q,P)=\frac{1}{2}\left(\beta P^2-\gamma Q^2\right) + \alpha Q\cdot P \ .
\end{equation}

\section{Newton's equation}

The Newton's equation (see \cite{arno}) is given by 
\begin{equation}
\left\{
\begin{array}{l l}
\dot{q}&=p/m, \\
\dot{p}&=-U'(q),
\end{array}
\right.
\label{ex2_eqcont}
\end{equation}
with $m\in \R^+$ and $q,p\in \R^d$. This equation possesses a natural Hamiltonian structure with a Hamiltonian given by
\begin{equation}
H(q,p)=\frac{1}{2m}p^2+U(q).
\end{equation}
Using the construction of Section 3.1, a natural discretisation is given by
\begin{equation}
\left\{
\begin{array}{l l}
\Delta{Q}&=P/m, \\
\nabla{P}&=-U'(Q),
\end{array}
\right.
\label{ex2_eq1}
\end{equation}
defined over $\TT^\pm$. The Hamiltonian Helmholtz conditions are clearly satisfied. 

\begin{remark}
It must be noted that the Hamiltonian associated to \eqref{ex2_eq1} is given by
\begin{equation}
H(Q,P)=\frac{1}{2m}P^2+U(Q),
\end{equation}
which can be recovered by formula \eqref{ham_form}.
\end{remark}
As shown in \cite{cp1} and Section 3.3, there is non-coherence between the differential and variational discrete embedding. Moreover we can detail the differences and consequences between these two approaches. Consider the $\Delta$ differential embedding of \eqref{ex2_eqcont} defined over $\TT^+$ which is given by 
\begin{equation}
\left\{
\begin{array}{l l}
\Delta{Q}&=P/m, \\
\Delta{P}&=-U'(Q),
\end{array}
\right.
\label{ex2_eq2}
\end{equation}
where $Q,P\in C(\TT,\R)$. Introducing $Z=\rho(P)$ we obtain the following equivalent system over $\TT^\pm$ given by
\begin{equation}
\left\{
\begin{array}{l l}
\Delta{Q}&=\rho(Z)/m, \\
\nabla{Z}&=-U'(Q).
\end{array}
\right.
\label{ex2_eq21}
\end{equation}
Now the problem is that the discrete field is depending on $Z$ through the shift operator $\rho$. As $\rho(Z) = Z - h \nabla(Z)$, inserting the definition of $\rho(Z)$ in $\eqref{ex2_eq21}$ we obtain
\begin{equation}
\left\{
\begin{array}{l l}
\Delta{Q}&=\frac{1}{m}\left(Z+hU'(Q)\right), \\
\nabla{Z}&=-U'(Q).
\end{array}
\right.
\label{ex2_eq22}
\end{equation}
In that case, the Helmholtz condition (CH2) is clearly satisfied. The condition (CH1) is equivalent to $U''(Q)=0$ which means that \eqref{ex2_eq22} is Hamiltonian if and only if U is linear in $Q$ which is a very restrictive condition.

\section{Linear friction}
Let us consider the continuous system with friction
\begin{equation}
\left\{
\begin{array}{l l}
\dot{q}&=p/m, \\
\dot{p}&=-\gamma p - q,
\end{array}
\right.
\label{ex3_eq}
\end{equation}
where $\gamma\in\R$ and $q,p\in \R^d$. As in the previous example, two choices are possible. The one mimicking the canonical form of the discrete Hamiltonian systems is given by
\begin{equation}
\left\{
\begin{array}{l l}
\Delta{Q}&=P/m, \\
\nabla{P}&=-\gamma P -Q,
\end{array}
\right.
\label{ex3_eq_v1}
\end{equation}
and the one corresponding to discrete differential embedding is given by
\begin{equation}
\left\{
\begin{array}{l l}
\Delta{Q}&=P/m, \\
\Delta{P}&=-\gamma P -Q .
\end{array}
\right.
\label{ex3_eq_v2}
\end{equation}

In the first case, the Hamiltonian Helmholtz condition (CH2) is clearly satisfied. The condition (CH1) is equivalent to $\gamma=0$. In other words, the system \eqref{ex3_eq_v1} has a discrete Hamiltonian structure if and only if there is no friction which is consistent with the continuous one. \\

In the second case, using the same trick as in the previous example, we obtain the following equivalent system
\begin{equation}
\left\{
\begin{array}{l l}
\Delta{Q}&=\rho(Z)/m, \\
\nabla{Z}&=-\gamma \rho(Z) - Q,
\end{array}
\right.
\label{ex3_eq_v21}
\end{equation}
where $Z=\rho(P)$. Assuming that $\gamma \neq 1/h$, we obtain
\begin{equation}
\left\{
\begin{array}{l l}
\Delta{Q}&=\frac{1}{m}\left(\frac{1}{1-h\gamma}Z+\frac{h}{1-h\gamma}Q\right), \\
\nabla{Z}&=-\frac{\gamma}{1-h\gamma}Z -\frac{1}{1-h\gamma}Q.
\end{array}
\right.
\label{ex3_eq_v22}
\end{equation}
As previously, there is only one condition to be satisfied which is (CH1). This condition is equivalent to $\gamma=\frac{h}{m}$. In other words, the system \eqref{ex3_eq_v1} has a discrete Hamiltonian structure if and only if the time step in the discretisation is equal to $m\gamma$.

\section{Modified Harmonic oscillator}

Let us consider the following discrete linear system defined over $\TT^\pm$ by
\begin{equation}
\left\{
\begin{array}{l l}
\Delta{Q}&=P + \alpha(P),\\
\nabla{P}&=Q.
\end{array}
\right.
\label{ex4_eq}
\end{equation}
where $Q,P\in C(\TT,\R^2)$ and $\alpha(Q)^\mathsf{T}=(P_2,0)$. In that case, the Helmholtz condition (CH2) is not satisfied. Indeed, $\frac{\partial X_Q}{\partial P}$ is not symmetric because $\frac{\partial \alpha(P)_1}{\partial P_2}=1$ and $\frac{\partial \alpha(P)_2}{\partial P_1}=0$. Consequently the system \eqref{ex4_eq} does not have a discrete Hamiltonian structure.

\part*{Conclusion and perspectives}

As in the Lagrangian case studied in \cite{cresson-bourdin}, one can generalize our result in the setting of the {\it time-scale calculus}. Contrary to what happens in the Lagrangian case where technical difficulties related to the composition of operators seem to cancel such a generalization, the Hamiltonian case is possible due to its linear nature in the discrete differential operators. We refer to \cite{pierret} for more details.\\

Another important task is to defined a notion of discrete differential forms adapted to our point of view in order to fully follow the continuous formulation of the integrability conditions in term of closed differential forms.\\

Last but not least, the previous approach can be used to solve the inverse problem of the calculus of variations in the stochastic setting introduced by J-M. Bismut in \cite{bismut}. In this book, he defined the notion of stochastic Hamiltonian systems and develop the corresponding stochastic calculus of variations providing all the ingredients to generalize Helmholtz's conditions to the stochastic case.
 
\bibliographystyle{plain}

\end{document}